\documentclass[12pt]{amsart}
\usepackage{a4wide}
\usepackage{times}
\usepackage{bbm}
\usepackage{mathtools, amssymb}
\usepackage{graphicx,xspace}
\usepackage{epsfig}
\usepackage{dsfont}
\usepackage[usenames,dvipsnames]{xcolor}
\usepackage{tikz}
\usepackage[T1]{fontenc}
\usepackage[utf8]{inputenc}
\usepackage{array,multirow} %
\usepackage{bm}
\usepackage{kpfonts} 
\usepackage{dsfont} 
\usepackage{setspace} 
\usepackage{cancel}
\usepackage{tipa}
\usepackage{stmaryrd}
\usepackage{csquotes}

\usepackage{hyperref,pifont}
\usepackage{todonotes}
\usepackage{dsfont}

\newcommand{\ignorer}[1]{}

\def\and{\ \wedge\ }
\usepackage[capitalize]{cleveref}

\newcommand{\revision}{}

\theoremstyle{plain}

\newtheorem{lemma}{Lemma}
\newtheorem{theorem}[lemma]{Theorem}

\newtheorem{proposition}[lemma]{Proposition}
\newtheorem{definition}[lemma]{Definition}

\theoremstyle{remark}
\newtheorem{remark}[lemma]{Remark}

\def\eps{\varepsilon}
\renewcommand\epsilon{\varepsilon}

\def\R{\mathbb{R}}

\def\P{\mathbb{P}}

\def\cH{\mathcal{H}}
\def\cP{\mathcal{P}}
\def\cL{\mathcal{L}}

\def\bbi{\overline{\bm i}}

\def\E{\mathbb{E}}
\def\Law{\mathcal{L}}

\def\Z{\mathbb{Z}}

\def\bM{\mathbb{M}}

\def\cF{\mathcal{F}}

\DeclareMathOperator{\Cat}{Cat}
\DeclareMathOperator{\IS}{IS}  
\DeclareMathOperator{\NC}{NC}  

\newcommand{\Old}[1]{}

\newcommand\restr[2]{{%
		\left.\kern-\nulldelimiterspace %
		#1 %
		\right|_{#2} %
	}}

\DeclareMathOperator{\cc}{cc}
\newcommand\esper{\mathbb E}

\title[Meanders and noodles]{Components in meandric systems and the infinite noodle}

 \author[V. Féray]{Valentin Féray}
       \address[VF]{Université de Lorraine, CNRS, IECL, F-54000 Nancy, France}
       \email{valentin.feray@univ-lorraine.fr}

 \author[P. Thévenin]{Paul Thévenin}
 \address[PT]{Ångström Laboratory, Lägerhyddsvägen 1, Uppsala}
       \email{paul.thevenin@math.uu.se}
 
     \keywords{meanders, local limits, infinite noodle}

\subjclass[2020]{60C05}

\begin{document}

\begin{abstract}
We investigate here the asymptotic behaviour of a large typical meandric system. More precisely, we show the quenched local convergence of a random uniform meandric system $\bm M_n$ on $2n$ points, as $n \rightarrow \infty$, towards the infinite noodle introduced by Curien, Kozma, Sidoravicius and Tournier ({\em Ann. Inst. Henri Poincaré D}, {6}(2):221--238, 2019). As a consequence, denoting by $cc( \bm M_n)$ the number of connected components of $\bm M_n$, we prove the convergence in probability of $cc(\bm M_n)/n$ to some constant $\kappa$, answering a question raised independently
by Goulden--Nica--Puder
({\em Int. Math. Res. Not.}, 2020(4):983--1034, 2020)
and Kargin ({\em Journal of Statistical Physics}, 181(6):2322--2345, 2020). 
This result also provides information on the asymptotic geometry of the Hasse diagram of the lattice of non-crossing partitions. Finally, we obtain expressions of  the constant $\kappa$ as infinite sums over meanders, which allows us to compute upper and lower approximations of $\kappa$.
\end{abstract}

\maketitle

\section{Introduction}

\subsection{Background and main result}
We study in this paper {\em meandric systems} of given size $n$, 
which are collections of non-crossing loops intersecting the horizontal axis exactly at the points $0, \ldots, 2n-1$ (up to continuous deformation fixing the horizontal axis). See Fig. \ref{fig:meander} for a simulation of a uniform meandric system of size $60$.

\begin{figure}
\center
\begin{tabular}{c c}

\includegraphics[scale=.6]{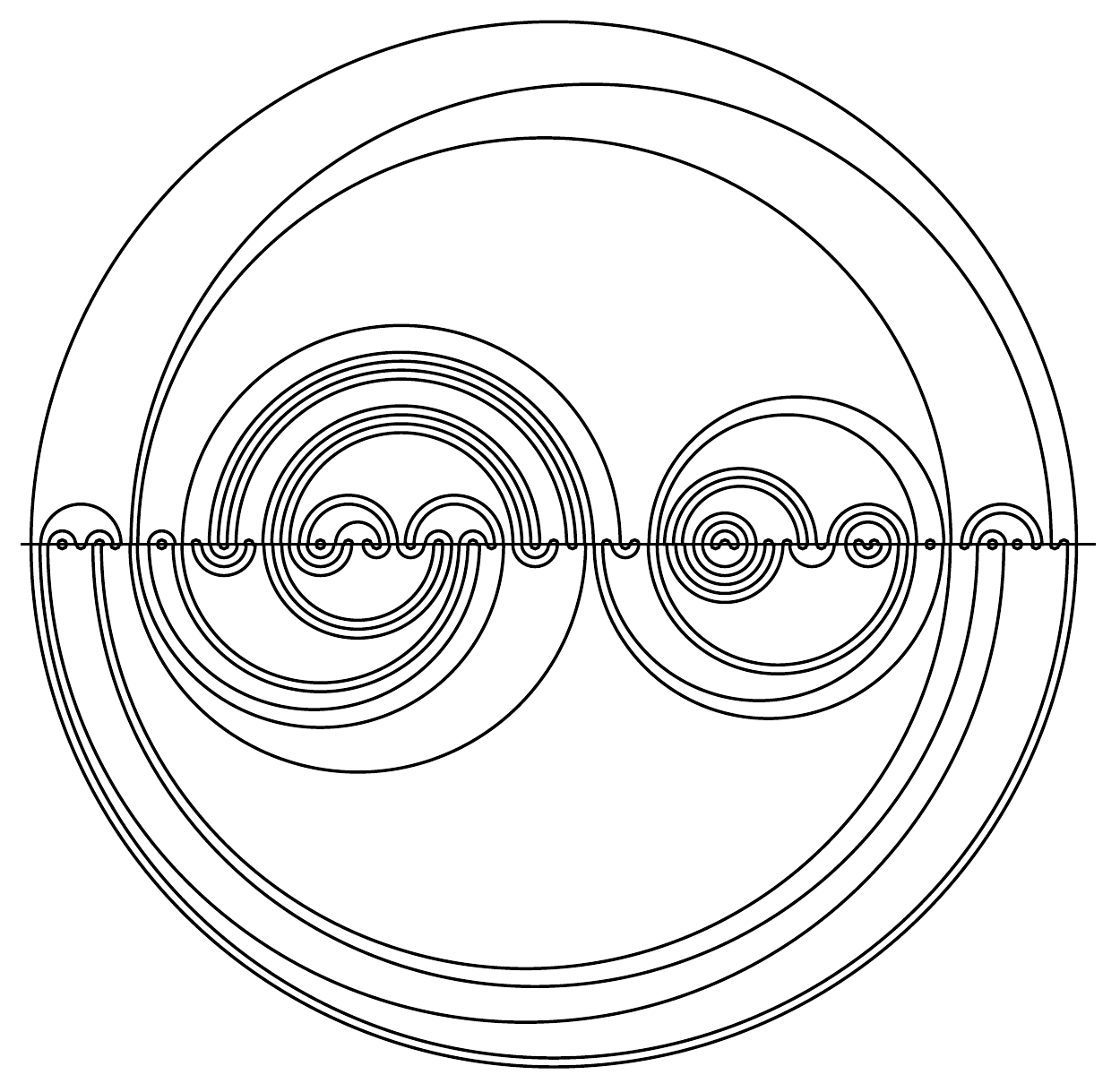}
&
\includegraphics[scale=.6]{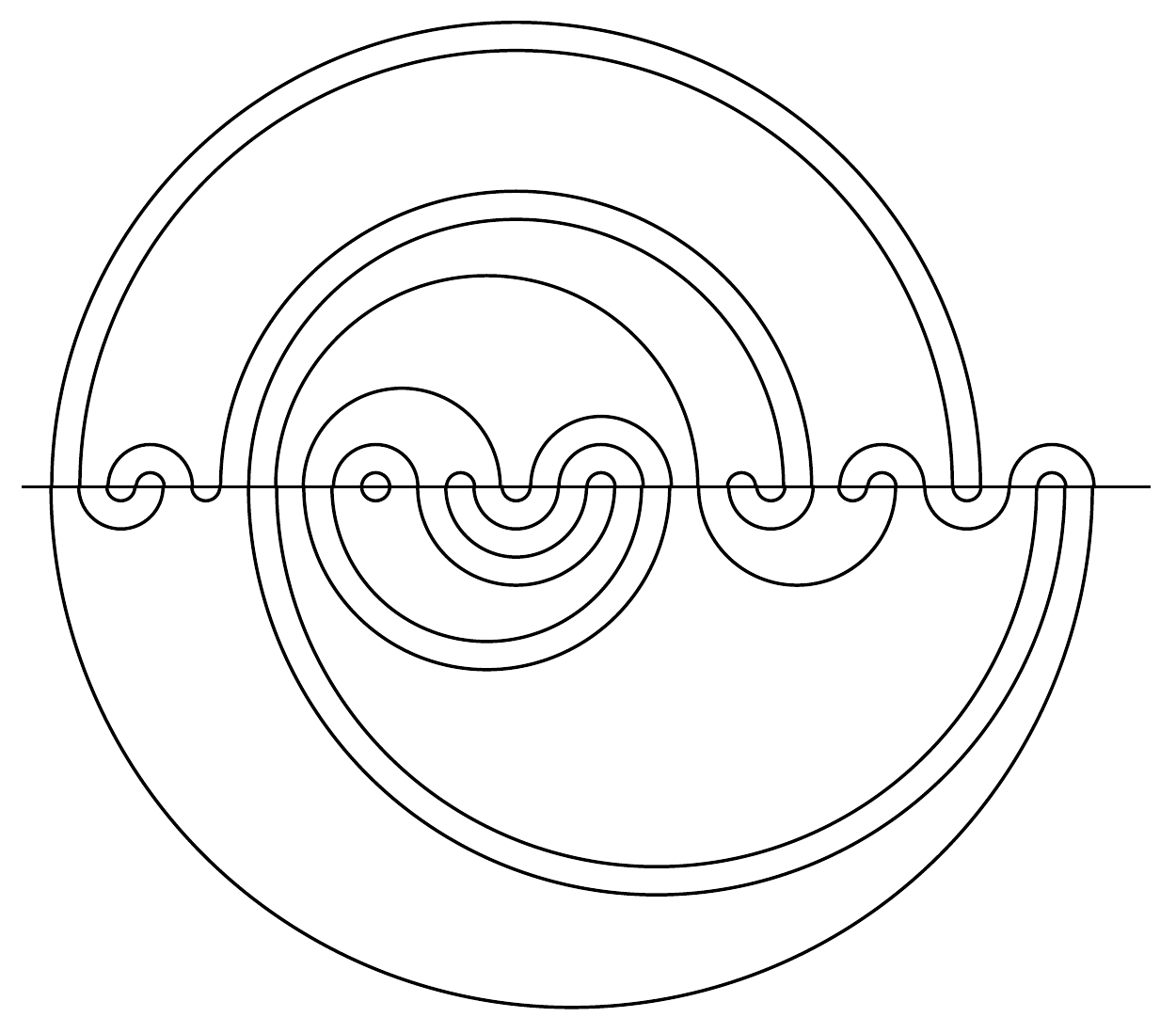}
\end{tabular}
\caption{Left: a uniform random meandric system of size $60$. Right: a uniform random meandric system of size $20$, with $4$ connected components.}
\label{fig:meander}
\end{figure}
 
A meandric system consisting of only one loop is called a {\em meander} (see Fig. \ref{fig:conme}).
Their study can be traced back to Poincaré, and they are connected to different domains of mathematics, theoretical physics or even biology, where they can be used as a model for polymer foldings \cite{DFGG97}.
\begin{figure}
\includegraphics[scale=.7]{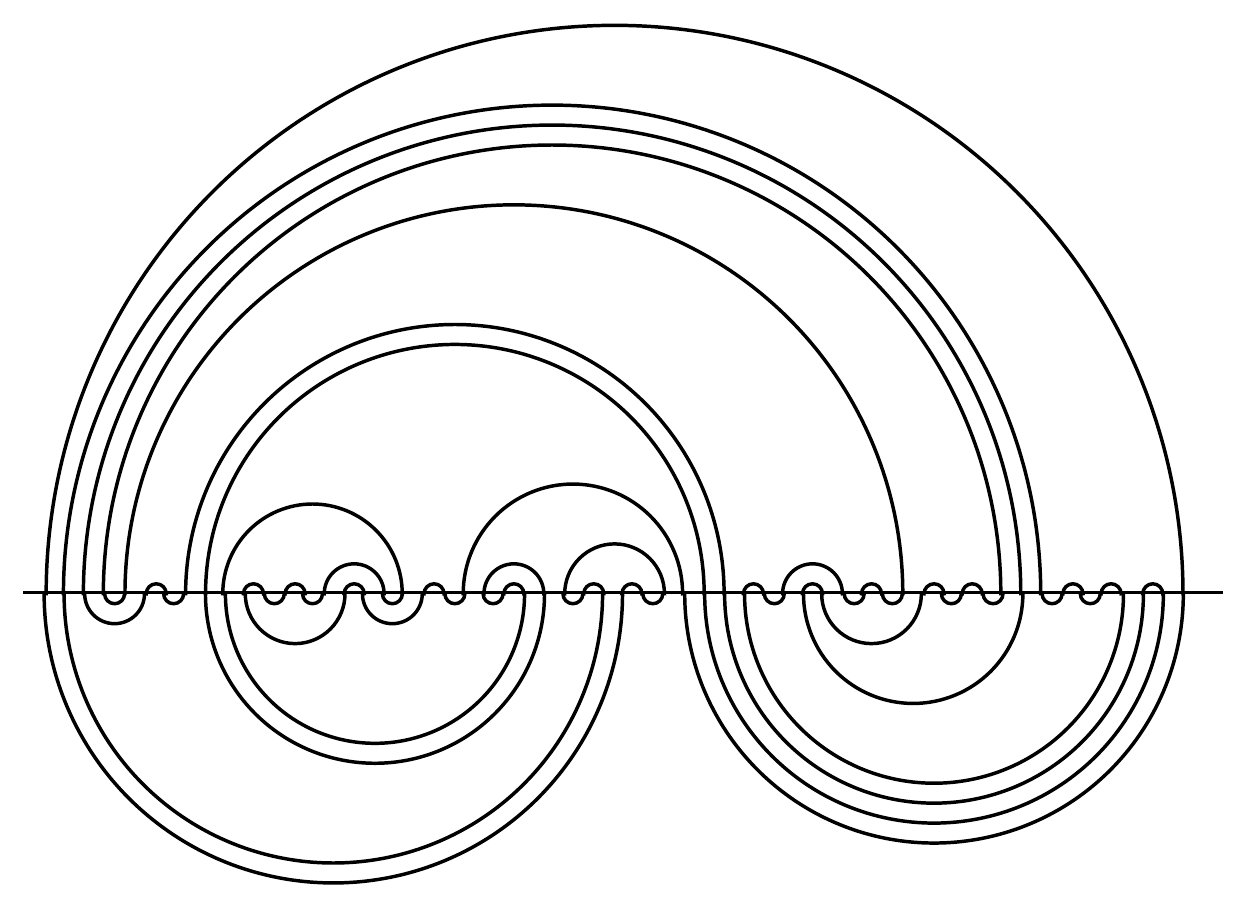}
\caption{A uniform random meander of size $30$.}
\label{fig:conme}
\end{figure}

Combinatorially, meandric systems can be uniquely represented 
as a pair of non-crossing pair-partitions ;
in particular there are $\Cat_n^2$ meandric systems of size $n$.
On the contrary,
the enumeration of meanders is a notoriously hard problem,
see \cite{zvonkine2021meanders} for a survey on the topic 
and \cite{difrancesco2000meanders} for a remarkable conjecture on the critical exponent. 
Subclasses of meandric systems or meanders have been also (at least asymptotically) enumerated, 
for example meandric systems with a large number of connected components \cite{FN19}
or meanders with a given number of minimal arcs \cite{DGZ20}.
This raises connections respectively with free probability theory (see also \cite{Nic16}) and 
with the work of Mirzakhani about enumeration of geodesics on surfaces \cite{Mir08}.

In this paper, we focus on another natural question: what is the number of connected components $\cc(\bm M_n)$ of a uniform random (unconditioned) meandric system  $\bm M_n$ of size $n$? This question has been raised recently, independently by Kargin~\cite{Kar20} and Goulden--Nica--Puder~\cite{GouldenMeandric}. 
Both sets of authors prove (through different methods) a linear lower bound for $\E[cc(\bm M_n)]$ and conjecture that the quotient $\E[\cc(\bm M_n)]/n$
converges to a constant. We show here a stronger version of this conjecture, proving the convergence in probability of $cc(\bm M_n)/n$ towards a constant.

\begin{theorem}\label{thm:cv_proba_ccMn}
Let $\bm M_n$ be a uniform random meandric system of size $n$.
Then there exists a constant $\kappa \in (0,1)$ such that $\cc(\bm M_n)/n \to \kappa$ in probability.
\end{theorem}

We note that Kargin~\cite{Kar20} further conjectures that $\cc(\bm M_n)$ is asymptotically normal with a variance linear in $n$, but we leave this problem open.

\subsection{Quenched local convergence to the infinite noodle}

A key step in our proof of \cref{thm:cv_proba_ccMn} is the identification of the 
local limit of the uniform meandric system $\bm M_n$ of size $n$.
This limit turns out to be the so-called {\em infinite noodle} $\bm M_\infty$,
introduced by Curien, Kozma, Sidoravicius and Tournier in \cite{InfiniteNoodle} (the term {\em infinite noodle} refers to a hypothetical infinite component of the system in \cite{InfiniteNoodle}, while we use here this denomination for the whole (multi)-graph on $\Z$, independently of the existence of an infinite component).
We recall here its construction. We first take two infinite sequences $(a_i)_{i \in \Z}$
and $(b_i)_{i \in \Z}$ of i.i.d.~balanced random variables in $\{L,R\}$.
Observe that there is a unique non-crossing
pair-partition of $\Z$ so that $i$ is the left end of an arc if and only if $a_i=L$.
We draw this non-crossing partition above the horizontal axis (\enquote{a} stands for \enquote{above}).
Similarly, we construct another non-crossing partition from the sequence $(b_i)_{i \in \Z}$
and draw it below the horizontal axis (\enquote{b} stands for \enquote{below}).
Taken together, the two matchings form a (multi-)graph on $\Z$; we are interested in its connected components.
 In particular it is proven in \cite{InfiniteNoodle} that, either there is exactly one infinite
 component a.s.,~or there is no infinite component a.s.
 The question of deciding which of these two statements holds is open.
\medskip

We now come back to the connection with meandric systems.
In \cref{sec:localcv}, a notion of local convergence for meandric systems is introduced.
The local convergence of a uniform meandric system $\bm M_n$ of size $n$
to the infinite noodle is rather intuitive, and proved in \cref{prop:CvToInfiniteNoodle} below,
based on some known results for conditioned random walks \cite{janson2016fringe, Bor21}.
This convergence holds in a strong sense, namely in the quenched setting
(see \cref{sec:localcv} for a definition).
 Roughly speaking, it says that, for any $k$ fixed, the $k$-neighbourhood of a uniform point 
 in a fixed typical meandric system $\bm M_n$ of size $n$ 
 converges in distribution to the $k$-neighbourhood of $0$ in $\bm M_\infty$. 
\medskip

As said above, this quenched local convergence result is central in our proof
of \cref{thm:cv_proba_ccMn}. In fact,
the constant $\kappa$ has a natural expression in terms of the infinite noodle:\begin{equation}\label{eq:kappa}
\kappa = \E \left[ \frac{2}{|C_0(\bm M_\infty)|} \right], 
\end{equation}
where $|C_0(\bm M_\infty)|$ is the size of the component containing $0$
in $\bm M_\infty$ (with the \revision{convention} $2/\infty=0$).
We note that it is not known whether this component is finite a.s.;
this is equivalent to the fact that the infinite noodle has no infinite component a.s., which, as we said before, is still an open question.

Based on this probabilistic interpretation,
in \cref{sec:formulas}, we give two formulas for $\kappa$,
as infinite sums over meanders of all sizes and closely related objects (open meanders).
Truncating these sums to finite index sets
yield rigourous lower and upper bounds for $\kappa$.
With the aid of a computer software, we can prove that $\kappa$
is in the interval $[0.207,0.292]$, improving on the bounds $0.17$ and $0.5$
given in \cite{GouldenMeandric} 
(computational experiments suggest that the true value should be around $0.23$, 
see again \cite{GouldenMeandric} or \cite{Kar20}).

\subsection{Asymptotic geometry of the Hasse diagram of the lattice of non-crossing partitions}
We now discuss a geometric consequence of our result,
regarding non-crossing partitions.
For brevity, we will not recall all definitions \revision{in the introduction}.
The reader unfamiliar with non-crossing partitions may look
at \revision{\cref{sec:appendix} or} \cite{GouldenMeandric} for the relevant definitions.
We denote by $\NC(n)$ the poset of non-crossing partitions of size $n$,
with the refinement order.
Considering the Hasse diagram (i.e.~the graph of covering relations)
 of this poset endows the set $\NC(n)$
with a graph structure; we will denote this graph by $\cH_n$.

It turns out that distances in this graph can be understood through meandric systems.
Indeed, in \cite{GouldenMeandric}, the authors construct a bijection $M$
from $\NC(n)^2$ onto the set of meandric systems of size $n$
such that, for all $\rho, \pi \in \NC(n)$:
\[d_{\cH_n}(\rho,\pi) = n- \cc\Big(M(\rho,\pi)\big)\Big).\]
We refer to \cref{sec:appendix} for a short informal description of the bijection.
Hence \cref{thm:cv_proba_ccMn} is equivalent to the following statement:
\begin{proposition}
\label{prop:limit_NCn}
For each $n$, let $\bm \rho_n$ and $\bm \pi_n$ be independent uniformly distributed elements in $\NC(n)$.
As $n$ tends to infinity, the quantity $\frac{1}{n}d_{\cH_n}(\bm \rho_n,\bm \pi_n)$ converges in probability to $1-\kappa$.
\end{proposition}
In other words, in the rescaled graph $(\cH_n, n^{-1} d_{\cH_n})$, asymptotically almost all pairs of points are at the same distance $1-\kappa$ from each other. 
We note that this kind of phenomena, where asymptotically all pairs of points
are at the same distance of each other, occurs in many other natural models
of (random) discrete measured spaces: for examples
in hypercubes of growing dimensions or
in models of random trees
of logarithmic height \cite[Sections 12 to 14]{janson2021tree_limits}.

\begin{remark}
  To understand further the asymptotic geometry of $\NC(n)$,
   a natural question is to study the distance to a fixed particular element,
   for example the partition $\bm 0_n$ in singletons (i.e.~the minimal element of $\NC(n)$).
Letting $\bm \rho_n$ a uniform random element in $\NC(n)$,
one can prove that  $\frac{1}{n}d_{\cH_n}(\bm \rho_n,\bm 0_n)$
converges to $1/2$ in probability.
Hence $\bm 0_n$ is not a typical element in $\NC(n)$.
The same holds replacing $\bm 0_n$ by the one-block partition $\bm 1_n$ by symmetry
(recalling that $\NC(n)$ is self-dual).
Here, we see a difference with the asymptotic geometry of 
random tree models with logarithmic height -- where the distance to the root
is typically half the distance between two random points --
and of hypercubes -- where the typical distance to the root is the same as between two uniform random vertices.
\end{remark}

\subsection{Combinatorial definition of meandric systems}
A matching of size $n$ is a partition of the set $\{0,\dots,2n-1\}$
into sets of size $2$. 
It is customary to represent a matching $P$ as a set of arcs all on the same side of a horizontal line,
with an arc connecting $i$ to $j$ for each pair $\{i,j\}$ in $P$.
A matching is non-crossing if the corresponding set of arcs does not contain any crossing,
i.e.~if there does not exist two pairs $\{i,k\}$ and $\{j,\ell\}$ in $P$, with $i<j<k<\ell$.
It is well-known that there are $\Cat_n$ matchings of size $n$,
where $\Cat_n:=\frac{1}{n+1} \binom{2n}{n}$  is the $n$-th Catalan number.
 
It is easy to see that
a meandric system $M$ of size $n$ can be uniquely represented as
 a pair of non-crossing matchings of the same size,
drawn respectively above and below the same horizontal line.
Consequently, there are exactly $\Cat_n^2$ meandric systems of size $n$.
Considering both non-crossing matchings together gives a (multi-)graph structure on $\{0,\dots,2n-1\}$
and we denote by $\cc(M)$ the number of connected components of this graph.
It corresponds to the number of loops when we see the meandric system
as a collection of non-crossing loops intersecting the horizontal axis exactly
at positions $0, \dots, 2n-1$.

\subsection*{Organization of the paper}

We show in Section \ref{sec:localcv} the local convergence of a random uniform meandric system to the infinite noodle (Proposition \ref{prop:CvToInfiniteNoodle}). We then use this result to prove Theorem \ref{thm:cv_proba_ccMn} in Section \ref{sec:mainproof}. Section \ref{sec:formulas} is devoted to the proof of two ways of writing the constant $\kappa$ of Theorem \ref{thm:cv_proba_ccMn}, whose computation provides upper and lower bounds for its value.
\revision{Finally, \cref{sec:appendix} provides definitions used in \cref{prop:limit_NCn}.}

\subsection*{Notation}

In the whole paper, $\overset{(d)}{\rightarrow}$ will denote the convergence in distribution of a sequence of random variables, and $\overset{\P}{\rightarrow}$ the convergence in probability.

\section{Quenched local convergence of a uniform meandric system}
\label{sec:localcv}

The goal of this section is to prove the convergence of a uniform meandric system of size $n$, in the so-called \textit{quenched Benjamini--Schramm sense}, towards the infinite noodle introduced by Curien, Kozma, Sidoravicius and Tournier \cite{InfiniteNoodle}. The proof of this result consists in encoding a meandric system by conditioned random paths, for which results of local convergence have been established \cite{janson2016fringe, Bor21}.

\subsection{Local topology on meandric systems}

Let us define a local topology on meandric systems with a marked point.
To this end, we introduce the notion of partial matchings and partial meandric systems.

Let $A$ be an integer interval (finite or infinite). A partial matching on $A$ is a partition
of $A$ into pairs and singletons, where singletons are decorated with either $L$ or $R$.
We represent partial matchings with a set of arcs, where pairs are treated as above
and a singleton $\{i\}$ decorated with $L$ (resp.~$R$) is represented by a dashed unbounded open arc,
whose left end (resp.~right end) is $i$.
A partial matching is \textit{non-crossing} if it does not contain any of the following configurations (see Fig. \ref{fig:forbiddenconfigurations}):
 \begin{itemize}
 \item two pairs $\{i,k\}$ and $\{j,\ell\}$,
with $i<j<k<\ell$;
\item a pair $\{i,k\}$ and a singleton $\{j\}$ with $i<j<k$ (regardless of the decoration of $\{j\}$);
\item two singletons $\{i\}$ and $\{j\}$ with $i<j$ where $i$ is decorated with $L$ and $j$ with $R$.
 \end{itemize}

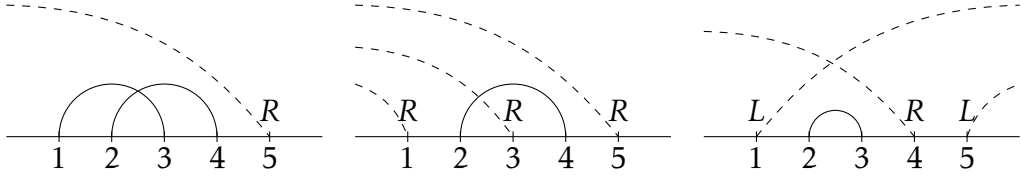
\begin{figure}
\[
\begin{tikzpicture}[scale=.7]
\draw (0,0) -- (6,0);
\draw (1,.1) -- (1,-.1) (2,.1) -- (2,-.1) (3,.1) -- (3,-.1) (4,.1) -- (4,-.1) (5,.1) -- (5,-.1);
\draw (1,-.4) node{$1$}; 
\draw (2,-.4) node{$2$}; 
\draw (3,-.4) node{$3$}; 
\draw (4,-.4) node{$4$}; 
\draw (5,-.4) node{$5$}; 
\draw (3,0) arc (0:180:1);
\draw (4,0) arc (0:180:1);
\draw (5,.5) node{$R$};
\draw[dashed] (0,2.5) to [bend left=25] (5,0);
\end{tikzpicture} 
\quad
\begin{tikzpicture}[scale=.7]
\draw (0,0) -- (6,0);
\draw (1,.1) -- (1,-.1) (2,.1) -- (2,-.1) (3,.1) -- (3,-.1) (4,.1) -- (4,-.1) (5,.1) -- (5,-.1);
\draw (1,-.4) node{$1$}; 
\draw (2,-.4) node{$2$}; 
\draw (3,-.4) node{$3$}; 
\draw (4,-.4) node{$4$}; 
\draw (5,-.4) node{$5$}; 
\draw (4,0) arc (0:180:1);
\draw (1,.5) node{$R$};
\draw[dashed] (0,1) to [bend left=25] (1,0);
\draw (3,.5) node{$R$};
\draw[dashed] (0,1.7) to [bend left=25] (3,0);
\draw (5,.5) node{$R$};
\draw[dashed] (0,2.5) to [bend left=25] (5,0);
\end{tikzpicture} 
\quad
\begin{tikzpicture}[scale=.7]
\draw[white] (0,2) -- (3,2);

\draw (0,0) -- (6,0);
\draw (1,.1) -- (1,-.1) (2,.1) -- (2,-.1) (3,.1) -- (3,-.1) (4,.1) -- (4,-.1) (5,.1) -- (5,-.1);
\draw (1,-.4) node{$1$}; 
\draw (2,-.4) node{$2$}; 
\draw (3,-.4) node{$3$}; 
\draw (4,-.4) node{$4$}; 
\draw (5,-.4) node{$5$}; 
\draw (3,0) arc (0:180:.5);
\draw (1,.5) node{$L$};
\draw[dashed] (1,0) to [bend left=25] (6,2.5);
\draw (4,.5) node{$R$};
\draw[dashed] (0,2) to [bend left=25] (4,0);
\draw (5,.5) node{$L$};
\draw[dashed] (5,0) to [bend left=25] (6,1);
\end{tikzpicture} 
\]
\caption{The three forbidden configurations in a non-crossing partial matching of the interval $\llbracket 1,5 \rrbracket$.}
\label{fig:forbiddenconfigurations}
\end{figure}

 A partial meandric system is then a pair of two non-crossing partial matchings of the same interval,
 drawn respectively above and below the same horizontal line.
 Components in partial meandric systems are defined similarly as in meandric systems.
We will sometimes refer to meandric systems as {\em complete} meandric systems
to emphasize the opposition to partial ones.

If $B$ is a subinterval of $A$ and $P$ a partial matching on $A$,
there is a natural notion of restriction $P/B$.
Namely, we keep only pairs and singletons of $P$ containing elements in $B$.
If a pair $\{i,i'\}$ of $P$ contains one element of $B$ and one not in $B$,
say $i$ is in $B$ but not $i'$, then $P/B$ contains $\{i\}$ as a singleton,
decorated either with $L$ (if $i' >i$) or with $R$ (if $i'<i$). See Fig. \ref{fig:restrictionmatching}.
Subsequently, we define the restriction $M/B$ of a partial meandric system $M$ on $A$
by taking the restrictions to $B$ of the two partial matchings of $M$.

\begin{figure}
\begin{tabular}{c}
\begin{tikzpicture}[scale=.8]
\fill [lightgray] (6.5,3) rectangle (11.5,-.7);
\draw (0,0) -- (17,0);
\draw (1,.1) -- (1,-.1) (2,.1) -- (2,-.1) (3,.1) -- (3,-.1) (4,.1) -- (4,-.1) (5,.1) -- (5,-.1) (6,.1) -- (6,-.1) (7,.1) -- (7,-.1) (8,.1) -- (8,-.1) (9,.1) -- (9,-.1) (10,.1) -- (10,-.1) (11,.1) -- (11,-.1) (12,.1) -- (12,-.1) (13,.1) -- (13,-.1) (14,.1) -- (14,-.1) (15,.1) -- (15,-.1) (16,.1) -- (16,-.1);
\draw (1,-.4) node{$-8$}; 
\draw (2,-.4) node{$-7$}; 
\draw (3,-.4) node{$-6$}; 
\draw (4,-.4) node{$-5$}; 
\draw (5,-.4) node{$-4$}; 
\draw (6,-.4) node{$-3$}; 
\draw (7,-.4) node{$-2$}; 
\draw (8,-.4) node{$-1$}; 
\draw (9,-.4) node{$0$}; 
\draw (10,-.4) node{$1$}; 
\draw (11,-.4) node{$2$}; 
\draw (12,-.4) node{$3$}; 
\draw (13,-.4) node{$4$}; 
\draw (14,-.4) node{$5$}; 
\draw (15,-.4) node{$6$}; 
\draw (16,-.4) node{$7$}; 
\draw (1,.5) node{$R$};
\draw[dashed] (1,0) to [bend right=25] (0,1);
\draw (3,0) arc (0:180:.5);
\draw (6,0) arc(0:180:.5);
\draw (8,0) arc(0:180:.5);
\draw (9,0) arc(0:180:2.5);
\draw (10,.5) node{$R$};
\draw[dashed] (10,0) to [bend right=55] (0,2.5);
\draw (16,0) arc(0:180:2.5);
\draw (15,0) arc(0:180:1.5);
\draw (14,0) arc(0:180:.5);
\end{tikzpicture}
\\
\begin{tikzpicture}
\draw[white] (0,0) -- (0,3);
\end{tikzpicture}

\begin{tikzpicture}
\draw (0.3,0) -- (6,0);
\draw (1,.1) -- (1,-.1) (2,.1) -- (2,-.1) (3,.1) -- (3,-.1) (4,.1) -- (4,-.1) (5,.1) -- (5,-.1);
\draw (1,-.4) node{$-2$}; 
\draw (2,-.4) node{$-1$}; 
\draw (3,-.4) node{$0$}; 
\draw (4,-.4) node{$1$}; 
\draw (5,-.4) node{$2$}; 
\draw (2,0) arc (0:180:.5);
\draw (3,.5) node{$R$};
\draw[dashed] (3,0) to [bend right=25] (0.3,1);
\draw (4,.5) node{$R$};
\draw[dashed] (4,0) to [bend right=25] (0.3,2);
\draw (5,.5) node{$L$};
\draw[dashed] (5,0) to [bend left=25] (6,1);
\end{tikzpicture}
\end{tabular}
\caption{A partial matching on $\llbracket -8,7 \rrbracket$ and its restriction to $\llbracket -2,2 \rrbracket$.}
\label{fig:restrictionmatching}
\end{figure}
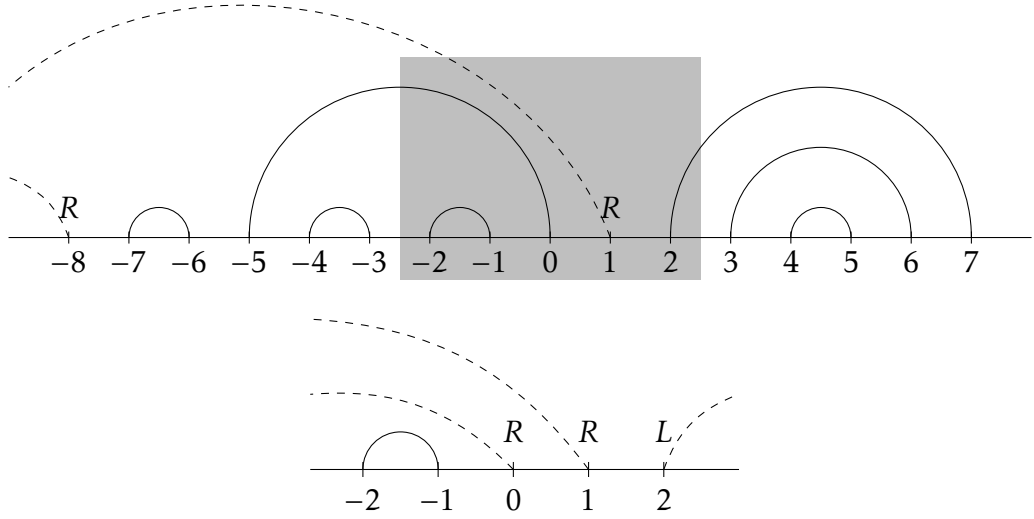

A marked (partial) meandric system  is a pair $(M,r)$,
where $M$ is a (partial) meandric system of some integer interval $A$ (possibly infinite)
and $r$ an element of $A$. We denote by $\bM_{part}$ the set of marked partial meandric systems.
Two marked partial meandric systems $(M_1,r_1)$ and $(M_2,r_2)$ 
are equivalent if they differ by a shift: we write $(M_1,r_1) \equiv (M_2,r_2)$.
Furthermore we define a distance between two marked partial meandric systems as follows
\[d\big( (M_1,r_1), (M_2,r_2) \big) = \frac{1}{1+\max\{k: M_1/[r_1-k;r_1+k] \equiv M_2/[r_2-k;r_2+k] \} },\]
with the convention that $d\big( (M_1,r_1), (M_2,r_2) \big)=0$ if they are equivalent and $1$ if $M_1 / \{r_1\} \neq M_2 / \{ r_2\}$.
With this distance, marked partial meandric systems, up to equivalence, form a compact Polish space.

\subsection{Benjamini--Schramm convergence}

In graph theory, we say that a (deterministic) sequence of graphs $(G_n)_{n \geq 1}$ converges in the Benjamini--Schramm sense \revision{- also known as local weak convergence -} if the sequence of random rooted graphs $(G_n,v_n)$ converges for the local topology,
where for each $n$, $v_n$ is a uniform random vertex of $G_n$.
This notion has first appeared in \cite{BS01} and has been largely used since then.
When the sequence of graphs $(G_n)$ is random, there are two natural extensions.
\begin{itemize}
\item Either we consider the pair $(G_n,v_n)$, where $v_n$ is a uniform random vertex of $G_n$, as a random rooted graph and we look at its limit in distribution.
In this case, we speak of {\em annealed Benjamini--Schramm convergence}.
\item Or we consider the conditional law $\mathcal L((G_n,v_n)|G_n)$,
which is a random measure on the set of rooted graphs.
If this random measure converges weakly \revision{ in distribution}, we speak of {\em quenched Benjamini--Schramm convergence}.
\end{itemize}
We will borrow this terminology in the setting of matchings and define a notion of quenched Benjamini--Schramm convergence. 
In the latter, we use the notation $\Law$ for the (conditional) law of a random variable.
\begin{definition}
A sequence of random partial meandric systems $(\bm M_n)_{n \geq 1}$ converges 
in the quenched Benjamini--Schramm sense 
if, letting $\bm i_n$ be a uniform random element of $\bm M_n$ \revision{(chosen independently of $\bm M_n$)}, we have \revision{ in distribution:} 
\[\Law\big((\bm M_n,\bm i_n) \, |\, \bm M_n\big) \overset{weakly}{\underset{n \rightarrow \infty}{\rightarrow}} \mu,\]
where $\mu$ is a random measure 
on the set of marked partial meandric systems.
Equivalently, it means that, for any bounded continuous function 
$F: (\bM_{part}, d) \rightarrow \R_+$, we have
\[ \E \left[ F(\bm M_n, \revision{\bm i_n}) | \bm M_n \right] \overset{(d)}{\underset{n \rightarrow \infty}{\rightarrow}} \int_{\bM_{part}} F(M,r) d\mu(M,r).\]
\end{definition}
To avoid dealing with random measures, we use mostly the second definition.
Remark moreover that, in the case of interest for us, the limiting measure is a.s.~equal to
the law of the infinite noodle, so that for all $F$:
\[ \int_{\bM_{part}} F(M,r) d\mu(M,r) = \E[ F(\bm M_\infty, 0) ] \]
is a deterministic quantity.

\subsection{Convergence of a uniform meandric system}

We prove here the following result, showing the quenched Benjamini--Schramm convergence 
of a uniform meandric system of size $n$ towards the infinite noodle $\bm M_\infty$.

\begin{proposition}
\label{prop:CvToInfiniteNoodle}
A uniform random meandric system $\bm M_n$ of size $n$ converges 
in the quenched Benjamini--Schramm sense 
to the law of the infinite noodle $(\bm M_\infty,0)$,
i.e.~for any bounded continuous function 
$F: (\bM_{part}, d) \rightarrow \R_+$, one has
\begin{align*}
\E \left[ F(\bm M_n, \revision{\bm i_n}) | \bm M_n \right] \overset{\P}{\underset{n \rightarrow \infty}{\rightarrow}} \E [ F(\bm M_\infty, 0) ].
\end{align*}
\end{proposition}

Since the right-hand side is a constant, observe that it is equivalent to the convergence in distribution towards the same quantity.
The proof is based on an straightforward encoding of meandric systems,
as pairs of well-parenthesized words, which we now describe.

Let $P$ be a partial non-crossing matching on an integer interval $A$.
We associate to it a function $w:A \to \{L,R\}$ as follows:
$w(i)=L$ if and only if $\{i\}$ is a singleton of $P$ decorated with $L$ or if $i$ is the smallest element in its pair. Such functions will be called {\em words} in the sequel.
The following lemma is elementary.
\begin{lemma}
This function $w$ defines a bijection from partial non-crossing matchings on $A$ to $\{L,R\}^A$,
compatible with restriction.
Moreover, for a finite interval $A$, a partial non-crossing matching is complete (i.e.~does not contain singletons)
if and only if the associated word is a well-parenthesized word (i.e.~it has as many $L$'s as $R$'s
and all its prefixes have at least as many $L$'s as $R$'s).
\end{lemma}
Using this correspondence, partial meandric systems are mapped to pairs $(w^a,w^b)$ of words in $\{L,R\}^A$. 
A pair $(w,r)$, where $w$ is a word in $\{L,R\}^A$, and $r$ a point in $A$, will be called a marked word.
Marked words are considered equivalent
if they differ only by a simultaneous shift of the domain set $A$
and of the marked point $r$.
We can define a natural distance between marked words $(w_1,r_1)$ and $(w_2,r_2)$ as follows:
\[D\big( (w_1,r_1), (w_2,r_2) \big) = \frac{1}{1+\max\{k: w_1/[r_1-k;r_1+k] \equiv w_2/[r_2-k;r_2+k] \} },\]
again with the convention that $D\big( (w_1,r_1), (w_2,r_2) \big)=0$ if they are equivalent and $1$ if $w_1(r_1) \neq w_2(r_2)$.
Again, this defines a compact metric space, which we denote $\mathcal W_*$.

We recall that the infinite noodle $\bm M_\infty$ is the meandric system on $\mathbb Z$ corresponding to random words $\bm w^a_\infty$, $\bm w^b_\infty$, where all images $\bm w^a_\infty(i)$ ($i \in \mathbb Z$) and $\bm w^b_\infty(j)$ ($j \in \mathbb Z$)
are uniform in $\{L,R\}$ and independent of each other.
On the other hand, a uniform meandric system $\bm M_n$ of size $n$ corresponds
to independent uniform well-parenthesized words $\bm w^a_n$, $\bm w^b_n$.
Since the bijection between partial matchings/meandric systems and words commutes with restriction,
\cref{prop:CvToInfiniteNoodle} is equivalent to the following proposition:
\begin{proposition}
\label{prop:LocalCvDyckpath}
A uniform random well-parenthesized word $\bm w_n$ on $\{0,\dots,n-1\}$
converges in the quenched Benjamini--Schramm sense
to the law of a random word $\bm w_\infty$, whose images $(\bm w_\infty(i))_{i \in \Z}$
are uniform in $\{L,R\}$ and independent of each other.
In other words, for any bounded continuous function $F: \mathcal W_* \rightarrow \R_+$, we have: 
\begin{equation}
\label{eq:ToProve}
\E \left[ F (\bm w_n,\bm i_n)  | \bm w_n \right] \overset{\P}{\underset{n \rightarrow \infty}{\rightarrow}} \E [ F(\bm w_\infty, 0) ],
\end{equation}
where $\bm i_n$ denotes a uniform element of $\{0, \ldots, n-1\}$.
\end{proposition}
\begin{proof}
As in the case of Proposition \ref{prop:CvToInfiniteNoodle}, it is enough to prove this convergence in distribution.
We first consider the case where $F$ depends only 
on the restriction of $w$ to a fixed size neighbourhood of $r$, say $F(w,r)=g(w(r-k),\dots,w(r+k))$ (we call such a function a local function).
Then 
\[ \E \left[ F (\bm w_n,\bm i_n)  | \bm w_n \right] =
\frac{1}{n} \sum_{r=k}^{n-k-1} g(\bm w_n(r-k),\dots, \bm w_n(r+k)) +O(1/n), \]
where the error term comes from boundary effect.
Taking $\bm w_n$ uniformly at random among all well-parenthesized word on $\{0, \ldots,n-1\}$
amounts to consider a simple random walk conditioned to come back at $0$ at time $n-1$ and be nonnegative between $0$ and $n-1$ (simply replace $L$ by $+1$ and $R$ by $-1$).
In this context it is known that 
\[\frac{1}{n} \sum_{r=k}^{n-k-1} g(\bm w_n(r-k),\dots, \bm w_n(r+k))\]
converges in probability to 
\[\mathbb E[g(\zeta_1,\dots, \zeta_{2k+1})],\]
where $(\zeta_1,\dots, \zeta_{2k+1})$ are i.i.d.random variables uniform in $\{-1;1\}$
(or in $\{L,R\}$): see \cite[Proposition $2.2$]{Bor21} for an asymptotic normality result implying this convergence in probability 
(Borga's result is itself based on \cite[Lemma 6.1]{janson2016fringe}).
Since 
\[\E [ F(\bm w_\infty, 0) ]= \mathbb E[g(\zeta_1,\dots, \zeta_{2k+1})],\]
this proves \eqref{eq:ToProve} for any {\em local} function $F$.

The general case follows by standard arguments, approximating $F$
by local functions. Here are the details.
We consider a general continuous function $F: \mathcal W_* \rightarrow \R_+$.
Since $\mathcal W_*$ is compact, $F$ is uniformly continuous.
For $k \ge 1$, we set $F_k(w,r)=F(w/[r-k;r+k],r)$.
We note that 
\[D((w,r),(w/[r-k;r+k],r)) \le 1/(k+1).\]
This bound is uniform on $(w,r)$ and implies, using the uniform continuity of $F$,
that $F_k$ tends uniformly to $F$.
We then write: a.s.,~it holds that
\begin{multline}
\label{eq:Triangular}
 \Big| \E \left[ F (\bm w_n,\bm i_n)  | \bm w_n \right] - \E [ F(\bm w_\infty, 0) ] \Big|
  \le \Big| \E \left[ F (\bm w_n,\bm i_n)  | \bm w_n \right] -
 \E \left[ F_k (\bm w_n,\bm i_n)  | \bm w_n \right] \Big|
 \\ + \Big|\E \left[ F_k (\bm w_n,\bm i_n)  | \bm w_n \right] - \E [ F_k(\bm w_\infty, 0) ] \Big| 
 +  \Big| \E [ F_k(\bm w_\infty, 0) ]- \E [ F(\bm w_\infty, 0) ] \Big|. 
 \end{multline}
 Fix $\eps>0$. Using the uniform convergence of $F_k$ to $F$,
 we can find $k$ such that the first and third terms are at most $\eps$ (a.s.~for the first term,
 which is a random variable).
 But, for this value of $k$, the second term converges to $0$ in probability,
 since we know that \eqref{eq:ToProve} holds for the local function $F_k$.
 We conclude that the left-hand side of \eqref{eq:Triangular} tends to $0$ in probability, as wanted.
\end{proof}

\section{Proof of \cref{thm:cv_proba_ccMn}}
\label{sec:mainproof}

We prove here our main result, Theorem \ref{thm:cv_proba_ccMn}, using Proposition \ref{prop:CvToInfiniteNoodle}. Indeed, it turns out that the number of connected components $cc(M)$ of a meander $M$ can be expressed as the expectation of
a functional of $(M,i)$, where $i$ is a random element of $M$.
To this end, if $M$ is a meandric system of size $n$ and $i$ an integer in $\{0,\dots,2n-1\}$,
we denote by $|C_i(M)|$ the size of the connected component $C_i(M)$ of $M$ containing $i$.

\begin{lemma}\label{lem:ccM_As_Expectation}
Let $M$ be a meandric system of size $n$ and $\bm i_n$ a uniform random integer in $\{0,\dots,2n-1\}$.
Then
$\tfrac{\cc(M)}{n} = \esper\big[ \tfrac{2}{|C_{\bm i_n}(M)|} \big]$.
\end{lemma}
\begin{proof}
Since $\bm i_n$ is uniformly distributed, we have
\[ \esper\big[ \tfrac{2}{|C_{\bm i_n}(M)|} \big] = \frac{2}{2n} \sum_{i=0}^{2n-1} \frac{1}{|C_i(M)|}.\]
Let $C$ be a connected component of $M$. Each $i$ in $C$ contributes $1/|C|$ to the above sum, so that all of them together contribute $1$ to the sum.
Since this holds for all connected components of $M$, the value of the sum
is simply the number of connected components of $M$, proving the lemma.
\end{proof}


We consider the following function on the set $\bM_{part}$ of marked partial meandric systems:
$\IS(M,r):=  \tfrac{1}{|C_r(M)|}$ ($\IS$ stands for \enquote{inverse size}).
Though not continuous on the whole space $\bM_{part}$,
this function has a nice continuity set.

\begin{lemma}
\label{lem:continuity_IS}
The functional $\IS$ is continuous at any {\em complete} meandric system. 
\end{lemma}
\begin{proof}
  Consider a sequence $(M_j,r_j)_{j \ge 1}$ converging to a {\em complete} meandric system $(M,r)$.
We want to prove that 
\begin{equation}
\label{eq:Continuity_IS}
 \lim_{j \to +\infty} \IS(M_j,r_j) = \IS(M,r).
\end{equation}
Assume first that the component $C_r(M)$ is finite. Then, since $M$ is complete, $C_r(M)$ is a loop. Let us denote by $k$ the longest distance between a point of $C_r(M)$ and $r$.
For $j$ large enough, we have $M_j / [r_j-k;r_j+k] =M / [r-k;r+k]$.
The latter contains entirely $C_r(M)$. We conclude that for $j$ large enough, $C_{r_j}(M_j)$ is also a loop, and has the same cardinality as $C_r(M)$. Hence, $\IS(M_j,r_j) = \IS(M,r)$, proving \eqref{eq:Continuity_IS} 
in the finite component case.

Assume now that $C_r(M)$ is infinite, so that $ \IS(M,r)=0$.
We fix an integer $m>0$. We can choose $k$ such that the component of $r$
in $M \cap [r-k;r+k]$ has at least $m$ points.
For $j$ large enough $M_j / [r_j-k;r_j+k] =M / [r-k;r+k]$,
so that the connected component of $r_j$ in $M_j$ contains at least $m$ points, i.e.~$\IS(M_j,r_j) \le 1/m$.
Since $m$ was arbitrary, we have
$ \lim_{j \to +\infty} \IS(M_j,r_j)=0$, proving \eqref{eq:Continuity_IS}
in the infinite component case.
\end{proof}

We can now prove our main theorem. 

\begin{proof}[Proof of \cref{thm:cv_proba_ccMn}]
Using Proposition \ref{prop:CvToInfiniteNoodle} along with \cite[Theorem $4.11$]{Kal17}, we know that
the sequence of random measures $\cL\big((\bm M_n, \bm i_n) | \bm M_n\big)$ converges in distribution for the vague topology towards the law of the infinite noodle $\cL\big(\bm M_\infty, 0\big)$. Thus, by \cite[Lemma $4.12$]{Kal17}, we have
\begin{equation}
\label{eq:cvofexpectation}
\E\big[ \IS(\bm M_n, \bm i_n) | \bm M_n \big] \overset{\P}{\underset{n \rightarrow \infty}{\rightarrow}} \E [\IS(\bm M_\infty, 0)],
\end{equation}
provided that the set of discontinuity of $\IS$ is negligible under the limiting measure $\cL\big(\bm M_\infty, 0\big)$. By Lemma \ref{lem:continuity_IS}, it is enough to prove that $\bm M_\infty$ is a.s.~a complete meandric system.
Assume without loss of generality that $w^a_\infty(0)=L$. Then, again with the convention that $L=1$ and $R=-1$, the random walk $(w^a_\infty(i))_{i \geq 1}$ is a.s.~recurrent. Let $j := \inf \{ i \geq 1, \sum_{k = 1}^i w^a_k=-1 \}$. Then $j<\infty$ a.s.,~and $0$ and $j$ are connected by an arc in the upper half-plane. The same works in the lower half-plane and, therefore, there is a.s.~no singleton in $\bm M_\infty$ and $\bm M_\infty$ is a.s.~complete. This proves Eq. \eqref{eq:cvofexpectation}.
By \cref{lem:ccM_As_Expectation}, the left-hand side of Eq. \eqref{eq:cvofexpectation} is $\tfrac1{2n} \cc(\bm M_n)$,
while $\kappa$ is twice the right-hand side (see \eqref{eq:kappa}).
Therefore, we have proved that $\tfrac1{n} \cc(\bm M_n)$ converges in probability to $\kappa$.
\end{proof}

\section{Two formulas for $\kappa$}
\label{sec:formulas}

It is possible to write the limiting constant $\kappa$ of Theorem \ref{thm:cv_proba_ccMn} in different ways. This allows in particular to bound from above and from below the value of this constant.

\subsection{Lower bounds}
Given a realization of the infinite noodle $\bm M_\infty$,
we will define the \textit{shape} $S_0$ of the component $C_0(\bm M_\infty)$ containing $0$. To obtain this shape, let $E \subset \Z$ be the set of points of $C_0(\bm M_\infty)$, and let $(\bm w^a_\infty, \bm w^b_\infty)$ be the words in $\{ L,R \}^{\Z}$ describing $\bm M_\infty$. 
Assuming $E$ is finite, there exists a unique increasing bijection $g: E \rightarrow \llbracket 0, |E|-1 \rrbracket$. The shape is the meandric system of size $|E|/2$ obtained from the pair $(\bm w^a, \bm w^b)$ on $\{ L,R \}^{|E|}$ such that, for all $0 \leq i \leq |E|-1$, $\bm w^a(i)= \bm w^a_\infty(g^{-1}(i))$ and $\bm w^b(i)=\bm w^b_\infty(g^{-1}(i))$. In particular, $S_0$ is a meandric system with only one connected component,
i.e.~a meander. It is therefore possible to write
\begin{align}\label{eq:decompo_kappa}
\kappa = \sum_{k=1}^{\infty} \frac{1}{k} \sum_{C \in M^{(1),k}} \P\left( S_0=C \right),
\end{align}
where $M^{(1),k}$ is the set of meanders of size $k$ (recall that they are defined on $\{ 0, \ldots, 2k-1 \}$).
 The interest of this formula is that, the meander $C$ being fixed, it is possible to compute $\P (S_0=C)$.
 To this end, we introduce some terminology.
Let $C$ be a meander of size $k$. We consider the union of $C$ and the real axis;
it divides the plane in some regions, two unbounded ones and several bounded ones.
We call such bounded regions {\em faces} of the meander, 
and denote their set by $\mathcal F(C)$. 
A face $F$ in $\mathcal F(C)$ is incident to a number of segments $[i;i+1]$;
we let $I(F)$ be the set of indices $i$ such that $F$ is incident to $[i;i+1]$.
A meander with six faces and the corresponding index sets
are shown on \cref{fig:meander_faces}.
With this notation, we have the following.
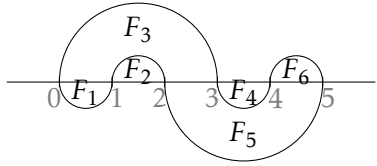
\begin{figure}
  \[
\begin{tikzpicture}[scale=.7,font=\small]
\draw (0,0) -- (7,0);
\draw (1,.1) -- (1,-.1) (2,.1) -- (2,-.1) (3,.1) -- (3,-.1) (4,.1) -- (4,-.1) (5,.1) -- (5,-.1) (6,.1) -- (6,-.1);
\draw (.9,-.3) node[gray]{$0$}; 
\draw (2.1,-.3) node[gray]{$1$}; 
\draw (2.9,-.3) node[gray]{$2$}; 
\draw (3.9,-.3) node[gray]{$3$}; 
\draw (5.1,-.3) node[gray]{$4$}; 
\draw (6.1,-.3) node[gray]{$5$}; 
\draw (3,0) arc (0:180:.5);
\draw (4,0) arc (0:180:1.5);
\draw (6,0) arc (0:180:.5);
\draw (1,0) arc (180:360:.5);
\draw (3,0) arc (180:360:1.5);
\draw (4,0) arc (180:360:.5);
\draw (1.5,-.2) node{$F_1$}; 
\draw (2.5,.2) node{$F_2$}; 
\draw (2.5,1) node{$F_3$}; 
\draw (4.5,-.2) node{$F_4$}; 
\draw (4.5,-1) node{$F_5$}; 
\draw (5.5,.2) node{$F_6$}; 
\end{tikzpicture} 
  \]
  \caption{A meander with six faces.
In this example, we have
$I(F_1)=\{0\}$, $I(F_2)=\{1\}$, $I(F_3)=\{0,2\}$, $I(F_4)=\{3\}$, $I(F_5)=\{2,4\}$ and $I(F_6)=\{4\}$.}
  \label{fig:meander_faces}
\end{figure}
\begin{proposition}
\label{lem:lowercomputation}
For any meander $C$ of size $k$, we have
\begin{equation}\label{eq:proba_meandre}
  \P (S_0=C) = 2^{-4k+1} k\, \sum_{\ell_1,\dots,\ell_{2k-1} \ge 0}  \left( \prod_{F \!\in\ \mathcal F(C)}
 \Cat_{\ell_{I(F)}} 2^{-2\ell_{I(F)}} \right),
 \end{equation}
 where, for a set $I$ of indices, we use the notation $\ell_I=\sum_{i\in I} \ell_i$.
\end{proposition}

\begin{proof}
Fix a meander $C$ of size $k$. 
For $1 \le i \le 2k$, we call $A_i$ the event \enquote{$0$ is the $i$-th smallest point
in $C_0(\bm M_\infty)$}. We clearly have
\[ \P (S_0=C) = \sum_{i=1}^{2k} \P \big[ (S_0=C) \and A_i \big].\]
By translation invariance of the model, the summands in the right-hand-side are all equal and
\begin{equation}\label{eq:proba_0AGauche}
  \P (S_0=C) = 2k\, \P \big[ (S_0=C) \and A_1 \big]. 
  \end{equation}
We split the event $\{(S_0=C) \and A_1\}$, depending on the set of points $E$
of $C_0(\bm M_\infty)$. Under the event $\{(S_0=C) \and A_1\}$, we have $|E|=2k$ and $\min(E)=0$,
so that $$E=\{0,m_1+1,m_1+m_2+2,\dots, m_1+\dots+m_{2k-1}+2k-1\},$$
for some integers $m_1,m_2,\dots,m_{2k-1} \ge 0$;
in other words, $m_i$ is the number of empty places (that is, integers that do not belong to $S_0$) between the $i$-th and $i\!+\!1$-st elements of $E$.
Necessarily, each $m_i$ must be even:
indeed since $C_0$ is connected, any other component of the meandric system
must cross the segment between the $i$-th and the $i\!+\!1$-st point of $E$
an even number of times.
We therefore write $m_i=2\ell_i$ for integers $\ell_1,\ell_2,\dots,\ell_{2k-1} \ge 0$.

Fixing some integers $\ell_1,\ell_2,\dots,\ell_{2k-1} \ge 0$, we compute the probability of the event
\begin{equation}
\label{eq:Event}
(S_0=C) \and A_1 \and \big(E=\{0,m_1+1,m_1+m_2+2,\dots, m_1+\dots+m_{2k-1}+2k-1\}\big).
\end{equation}

We claim that the probability of this event is given by
\begin{equation}
  \label{eq:Proba}
  2^{-4k}  \prod_{F \!\in\ \mathcal C} \Cat_{\ell_{I(F)}} 2^{-2\ell_{I(F)}}\revision{.}
\end{equation}

To avoid complicated notation, we will prove this formula on an example.
We consider the meander drawn on \cref{fig:meander_faces}.
We recall that the infinite noodle is constructed from two independent
random words $\bm w^a_\infty$ and $\bm w^b_\infty$ in the alphabet $\{L,R\}$ (the upper and lower words below;
we shall also think of $L$ and $R$ as arrows pointing to the left or right, when convenient).
Then the event \eqref{eq:Event} holds if and only if the words $\bm w^a_\infty$ and $\bm w^b_\infty$ satisfy the following properties.
\begin{itemize}
  \item All upper and lower arrows at points $\{0,m_1+1,m_1+m_2+2,\dots, m_1+\dots+m_{2k-1}+2k-1\}$
    should be pointing in a specific direction so that the infinite noodle restricted to this set
    can have shape $C$. In our example, it means that at $0$, both arrows should point to the right,
    at $m_1+1$, the upper arrow points to the right, and the lower to the left and so on\dots
    This happens with probability $2^{-4k}$ (here $k=3$).
  \item To have an arc in the lower matching between $0$ and $m_1+1$,
    we need to have a well-parenthesized word in $L$ and $R$ at positions $\{1,\dots,m_1\}$
    in the lower word $\bm w^b_\infty$.
    This happens with probability $\Cat_{\ell_1} 2^{-2\ell_1}$,
    where $m_1=2\ell_1$.
    Similarly, to have $m_1+1$ and $m_1+m_2+2$ connected in the upper matching,
    we need to have a well-parenthesized word in $L$ and $R$ at positions $\{m_1+2,\dots,m_1+m_2+1\}$.
    This happens with probability $\Cat_{\ell_2} 2^{-2\ell_2}$.
    Conditioning on having an arc between $m_1+1$ and $m_1+m_2+2$ in the upper matching,
    we also want $0$ and $m_1+m_2+m_3+3$ to be connected.
    This happens if and only if the restriction of $\bm w^a_\infty$ to $$\{1,\dots,m_1,m_1+m_2+3,\dots,m_1+m_2+m_3+2\}$$ is well-parenthesized;
    this is an event of probability $\Cat_{\ell_1+\ell_3} 2^{-2\ell_1-2\ell_3}$. 
    Continuing this way, we see that 
    we get one term $\Cat_{\ell_{I(F)}} 2^{-2\ell_{I(F)}}$ for each face $F$ of the matching.
\end{itemize}
This proves that the probability of the event \eqref{eq:Event} is indeed given by \eqref{eq:Proba}
and ends the proof of the lemma.
\end{proof}

Using the lemma and Mathematica to evaluate the infinite sums,
it is possible to evaluate $\P (S_0=C)$ for simple meanders.
\begin{itemize}
  \item for $C=\begin{array}{c}
\begin{tikzpicture}[scale=.4,font=\small]
\draw (0.7,0) -- (2.3,0);
\draw (1,.1) -- (1,-.1) (2,.1) -- (2,-.1) ;
\draw (2,0) arc (0:360:.5);
\end{tikzpicture}\end{array}$ (which is the only meander of size 2),
we have
\[\P (S_0=C) = \frac{1}{8} \sum_{\ell=0}^\infty \Cat_\ell^2 2^{-4\ell} = \frac{2}{\pi}-\frac{1}{2} \approx 0.137\]
\item for $C=\begin{array}{c}                                 
\begin{tikzpicture}[scale=.3,font=\small]                       
\draw (0.7,0) -- (4.3,0);                                       
\draw (1,.1) -- (1,-.1) (2,.1) -- (2,-.1) (3,.1) -- (3,-.1) (4,.1) -- (4,-.1) ;
\draw (4,0) arc (0:180:1.5);                                     
\draw (3,0) arc (0:180:.5);
\draw (1,0) arc (180:360:.5);                                     
\draw (3,0) arc (180:360:.5);                                     
\end{tikzpicture}\end{array}$
(which is the only meander of size $4$, up to vertical symmetry), we have

\begin{align}
  \P(S_0=C)&=\frac{1}{64} \cdot \left( \sum_{\ell_2 \ge 0} \Cat_{\ell_2} 2^{-2\ell_2}\right)
\cdot  \left(  \sum_{\ell_1,\ell_3 \ge 0}
 \Cat_{\ell_1}\Cat_{\ell_3} \Cat_{\ell_1+\ell_3}
    2^{-4\ell_1-4\ell_3} \right)   \nonumber\\
    &=\frac{1}{64} \cdot 2 \cdot \left( 8 -\frac{64}{3\pi} \right)
    =\frac14-\frac2{3\pi} \approx 0.038 \label{eq:facto_example}
  \end{align}
\end{itemize}
Probabilities $\P(S_0=C)$ for larger meanders $C$
do not seem to have such compact expressions in $\mathbb Q(\pi)$.

\begin{remark}
The factorization occurring in \cref{eq:facto_example} is an instance of a general fact.
Indeed, each face $F$ of a meander $C$ contains either only even numbers
or only odd numbers. 
We write $\mathcal F_e(C)$ and $\mathcal F_o(C)$ for the sets of faces
of $C$ with even and odd numbers, respectively. We have
\begin{multline*}\sum_{\ell_1,\dots,\ell_{2k-1} \ge 0}  \left( \prod_{F \!\in\ \mathcal F(C)}
 \Cat_{\ell_{I(F)}} 2^{-2\ell_{I(F)}} \right)
 = \left(\sum_{\ell_1,\ell_3,\dots,\ell_{2k-1} \ge 0} 
  \prod_{F \!\in\ \mathcal F_o (C)}
 \Cat_{\ell_{I(F)}} 2^{-2\ell_{I(F)}} \right)  \\
 \cdot
\left(\sum_{\ell_2,\ell_4,\dots,\ell_{2k-2} \ge 0} 
  \prod_{F \!\in\ \mathcal F_e (C)}
 \Cat_{\ell_{I(F)}} 2^{-2\ell_{I(F)}} \right).
 \end{multline*}
 This factorization yields sums over smaller index sets,
 which are easier to approximate numerically than the original sum.
\end{remark}

Plugging \cref{eq:proba_meandre} into \cref{eq:decompo_kappa}, 
and using the above remark, gives the following formula for $\kappa$:
\begin{multline*}
  \kappa=\sum_{k=1}^{\infty} 2^{-4k+1}\sum_{C \in M^{(1),k}}
\left(\sum_{\ell_1,\ell_3,\dots,\ell_{2k-1} \ge 0} 
  \prod_{F \!\in\ \mathcal F_o (C)}
 \Cat_{\ell_{I(F)}} 2^{-2\ell_{I(F)}} \right) \\
 \cdot
\left(\sum_{\ell_2,\ell_4,\dots,\ell_{2k-2} \ge 0} 
  \prod_{F \!\in\ \mathcal F_e (C)}
 \Cat_{\ell_{I(F)}} 2^{-2\ell_{I(F)}} \right).
 \end{multline*}
Replacing infinite sums by some truncations allows to find lower bounds on $\kappa$.
We have implemented this formula in Sage \cite{SageMath} (the code is available upon request) and found that $\kappa \ge 0.207$.

\subsection{Upper bounds}

In order to obtain upper bounds on $\kappa$, remark first that we can rewrite its expression in terms of the following event
\enquote{the point $0$ is the leftmost point of the connected component $C_0(\bm M_\infty)$ in the infinite noodle}.
We call this event $L_0$.

\begin{lemma}
\label{lem:probagauche}
We have
\begin{align*}
\kappa := \E \left[\frac{2}{|C_0(\bm M_\infty)|} \right] = 2 \, \P \left( L_0 \right)
\end{align*}
\end{lemma}

\begin{proof}
From \cref{eq:proba_0AGauche}, we know that, conditionally on the size $|C_0(\bm M_\infty)| \in \mathbb N \cup \{\infty\}$ of the connected component of $0$, the probability that $0$ is the leftmost point of the component is exactly $1/|C_0(\bm M_\infty)|$
(the argument is given for $|C_0(\bm M_\infty)|<+\infty$, but is easily adapted to the infinite case
to show that infinite components, if they exist, do not have leftmost points a.s.~).
This implies the lemma.
\end{proof}

In order to see whether $0$ is the leftmost point in its component or not, we discover the component step by step. More precisely, for any $k \geq 0$, define the event $E_k$: $0$ is not the leftmost point of its connected component restricted to $(-\infty,k \rrbracket$. In other words, one of the paths starting from $0$ visits $(-\infty, -1 \rrbracket$ before visiting $\llbracket k+1,+\infty)$. Observe that the event $E_k$ only depends on the value of $(\bm w^a_\infty, \bm w^b_\infty)$ on $\llbracket 0, k \rrbracket$. 
Clearly, denoting $\overline{L_0}$ the complement of the event $L_0$,
we have $\overline{L_0}=\bigcup_{k \ge 0} E_k$.
Assuming $\overline{L_0}$, one can thus define
\begin{align*}
K := \inf \{ k \geq 0, E_k \}.
\end{align*}
In particular, the event $E_0$ holds if and only if either $\bm w^a_\infty(0)$ or $\bm w^b_\infty(0)$ is $R$ (which happens with probability $3/4$).
Suppose now that $K\ge 1$.
Then we define the \textit{partial shape} $P_0$ of the component of $0$ as follows: let $E \subset \Z_+$ be the connected component of $0$ in the restriction $\bm M_\infty / \llbracket 0;K \rrbracket$. There exists a unique bijection $g: E \rightarrow \llbracket 0, |E|-1 \rrbracket$. The partial shape is the partial meandric system on $\llbracket 0, |E|-1 \rrbracket$ obtained from the pair $(\bm w^a, \bm w^b)$ on $\{ L,R \}^{|E|}$ such that, for all $0 \leq i \leq |E|-1$, we set $\bm w^a(i)= \bm w^a_\infty(g^{-1}(i))$ and $\bm w^b(i)=\bm w^b_\infty(g^{-1}(i))$. In particular, $P_0$ is a partial meandric system with only one connected component.
Moreover, since the event $E_K$ holds, $P_0$ has at least one singleton either in the lower or upper matching
(coming from an arc connecting $\llbracket 0;K \rrbracket$ to $(-\infty, -1 \rrbracket$).

\begin{lemma}
\label{lem:prop_P0}
Assume that $\overline{L_0}$ holds.
Then, $P_0$ satisfies the following conditions.
\begin{itemize}
\item[(i)] either $K=0$, or the restriction of $P_0$ to $\llbracket 0, |E|-2 \rrbracket$ does not have
$0$ and a singleton $R$ in the same connected component.
 \item[(ii)] Each of the words $\bm w^a$ and $\bm w^b$ of $P_0$ has exactly one singleton, one of them being $L$ and the other $R$.
\item[(iii)] Necessarily, $P_0$ has odd size.
\end{itemize}
\end{lemma}

\begin{proof}
All three items are straightforward if $K=0$. Let us therefore assume in the proof that $K \geq 1$.

Item (i) follows by minimality of $K$. Indeed, if item (i) is not satisfied,
then $0$ is connected to $(-\infty, -1 \rrbracket$ in the restriction  $\bm M_\infty / \llbracket 0;K-1 \rrbracket$, contradicting the minimality of $K$.

Let us first prove (ii). Since $P_0$ is connected and has at least one singleton, then necessarily the total number of singletons is exactly $2$. Except if both are at position $0$, it is clear that they have to be at different positions in $P_0$, say $s_1$ and $s_2$. 
Furthermore, necessarily one of them, say $s_1$
 (we do {\bf not} assume $s_1 <s_2$), has to take the value $R$ by definition.
Then, by item (i), $|E|-1$ should be between $s_1$ and $0$ when we follow $P_0$.
Therefore, it cannot be between $s_2$ and $0$ ($s_1$ and $s_2$ are the extremities of $P_0$), so that $s_2$ is already connected to $0$ in $P_0$ restricted to $\llbracket 0, |E|-2 \rrbracket$. Using item (i), we conclude that the singleton in $s_2$ has value $L$.

We still need to prove that $s_1$ and $s_2$ are not singletons 
on the same side of the real line.
Assume it is the case. Then $P_0$ visits first $s_1$ with the value $R$,
then $|E|-1$, then $0$ and finally $s_2$ with the value $L$ on the same side of the real line.
This is impossible without crossing arches. Thus, (ii) holds.

 Item (iii) follows directly, since matchings with exactly one singleton necessarily have odd size.
\end{proof}

For any $k \geq 1$, we denote by $\cP_k$ the set of partial meanders on $\llbracket 0, 2k \rrbracket$ satisfying the three conditions of \cref{lem:prop_P0}.
The set $\cP_k$ is a subset of what are usually called open meanders of odd size.
We can write:

\begin{proposition}
\label{prop:formulaUpper}
\begin{equation}\label{eq:UpperBound}
1-\frac{\kappa}{2} = \P\big( \, \overline{L_0}\, \big)
=\frac{3}{4} + \sum_{k \geq 1} \sum_{P \in \cP_k} \P \left( P_0=P\right).
\end{equation}
\end{proposition}
The term $3/4$ corresponds to the probability that $K=0$.
We note that keeping only this term gives $1-\frac\kappa{2} \ge \frac34$,
 i.e.~$\kappa \le \frac12$,
which is exactly the upper bound provided by Goulden, Nica and Puder \cite{GouldenMeandric}. 

Similarly to $\P \left( S_0=C \right)$ above, we can obtain a summation formula for
 $\P \left( P_0=P \right)$. 
 Let $P$ be in $\cP_k$. As for meanders, we denote $\mathcal F(P)$ its set of faces,
 i.e. {\em bounded regions} delimited by the arcs and the real line.
 We  also define two {\em open faces} as follows.
 Assume (up to symmetry) that there exists $s_1, s_2 \leq 2k$ such that $\bm w^a(s_1)=L$ and $\bm w^b(s_2)=R$.
 We complete $P$ with an arc in the upper half-plane from $s_1$ to $2k+1$
 and an arc in the lower half-plane from $-1$ to $s_2$.
 This creates two new faces that we denote respectively $F_\infty$ and $F_0$. See Fig. \ref{fig:partialshape}.

\begin{figure}
\begin{tikzpicture}[scale=.8,font=\small]    
\draw (-1,-.5) node{$-1$};   
\draw (0,-.5) node{$0$};   
\draw (1,-.5) node{$1$};   
\draw (2,-.5) node{$2$};  
\draw (3,-.5) node{$3$}; 
\draw (4,-.5) node{$4$}; 
\draw (5,-.5) node{$5$}; 
\draw (6,-.5) node{$6$}; 
\draw (7,-.5) node{$7$}; 
\draw (8,-.5) node{$8$}; 
\draw (9,-.5) node{$9$};                    
\draw (-1.5,0) -- (9.5,0);                                       
\draw (1,.1) -- (1,-.1) (-1,.1) -- (-1,-.1) ;
\draw (1,0) arc (0:-180:.5);
\draw (2,0) arc (0:180:.5);
\draw (4,0) arc (0:180:.5);
\draw (5,0) arc (0:-180:.5);
\draw (5,0) arc (0:180:2.5);
\draw (6,0) arc (0:-180:1.5);
\draw (7,0) arc (0:-180:2.5);
\draw (8,0) arc (0:180:.5);
\draw[dotted] (8,0) arc (0:-180:4.5);
\draw[dotted] (9,0) arc (0:180:1.5);
\draw[->] (8,0) arc (0:-45:4.5);
\draw[->] (6,0) arc (180:90:1.5);

\draw (6.8,.6) node{$F_\infty$};
\draw (1.5,-3) node{$F_0$}; 

\end{tikzpicture}
\caption{A possible partial shape on $\llbracket 0, 8 \rrbracket$ and the two added virtual open faces $F_0$ and $F_\infty$.}
\label{fig:partialshape}
\end{figure}
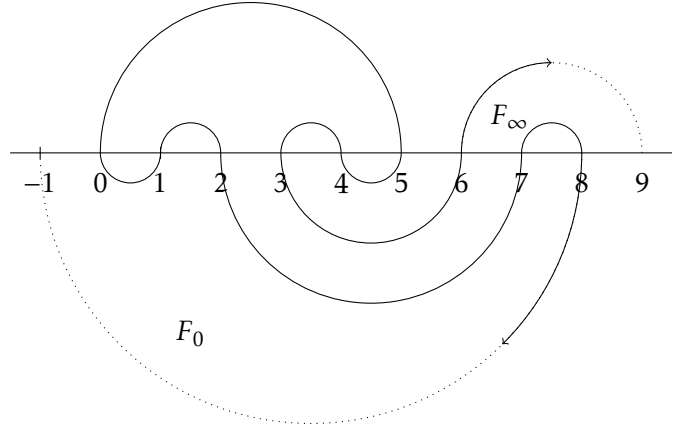

Before stating our formula, we recall a classical result of enumerative combinatorics (see e.g. \cite[Ex.$6.19$, p.$219$]{Sta99}).
For any $n \geq 0$, we denote by $N_n$ the number of simple walks $S: \llbracket 0, n \rrbracket \rightarrow \Z$ such that $S_0=0$ and $S \geq 0$. 

\begin{lemma}
\label{thm:stanley}
For all $n \geq 1$, we have
\begin{align*}
N_n = \binom{n}{\lfloor n/2 \rfloor}.
\end{align*}
\end{lemma}

This helps us get the following:

\begin{proposition}
\label{lem:uppercomputation}
For any $k \geq 1$ and any element $P \in \cP_k$, we have (keeping the notation of  \cref{lem:lowercomputation})
\begin{align*}
\P \left( P_0 = P \right) = 2^{-4k-2} \sum_{\ell_0, \ldots, \ell_{2k-1} \geq 0} \left( \prod_{F \in \cF} Cat_{\ell_{I(F)}} 2^{-2\ell(F)} \right) \times  \binom{2\ell_{I(F_0)}}{\ell_{I(F_0)}} \, 2^{-2\ell_{I(F_0)}}  \times  \binom{2\ell_{I(F_{\infty})}}{\ell_{I(F_{\infty})}} \, 2^{-2\ell_{I(F_{\infty})}}
\end{align*}
where, by convenience, we have set $\ell_{-1}=\ell_{2k}=0$.
\end{proposition}

\begin{proof}
We prove this lemma the same way as \cref{lem:lowercomputation}. Fix $k \geq 1$, and take $P$ a partial meander in $\cP_k$. We split the event $P_0=P$, depending on the set of points $E$. As in the proof of \cref{lem:lowercomputation}, we can find integers $m_1, \ldots, m_{2k-1} \geq 0$ such that $E = \{ 0, m_1+1, \ldots, m_{2k-1}+2k-1 \}$, all these integers being even. For all $i \leq 2k-1$, set $m_i=2\ell_i$.

The integers $\ell_1, \ldots, \ell_{2k-1}$ being fixed, we compute the probability of the event
\begin{equation}
\label{eq:eventupperbound}
(P_0=P) \wedge \left( E = \{ 0, m_1+1, \ldots, m_{2k-1}+2k-1 \} \right).
\end{equation}

The factors corresponding to the bounded faces that are not $F_0$ nor $F_{\infty}$ are obtained the same way as in \cref{lem:lowercomputation}. Let us focus on $F_0$. Assume that the singleton $R$ in $P_0$ is in the lower matching, in position $s_1 \in E$. In order for \eqref{eq:eventupperbound} to hold, then $s_1$ must not be connected by $\bm w^b$ to any point of $\llbracket 0,s_1-1 \rrbracket$. Conditionally on other arcs being present, this holds if and only if $s_1$ is not connected by $\bm w^b$ to any other point of $F_0$. For any $i \leq |F_0|$, let $S_i$ be the number of points $x$ among the \textit{last} $i$ points of $F_0$ satisfying $\bm w^b(x)=R$. Then, $s_1$ is not connected to any point of $F_0$ by $\bm w^b$ if and only if $S_i \geq 0$ for all $i \leq |F_0|$.
The term corresponding to $F_0$ follows by \cref{thm:stanley}.
The term corresponding to $F_\infty$ is obtained the same way.
\end{proof}

As an illustration, we make explicit the term corresponding to $k=1$ in \cref{prop:formulaUpper}.
There are two partial matchings in $\cP_1$, which are symmetric from each other
 through the reflection over the $x$-axis.
The corresponding terms in \cref{prop:formulaUpper} are given by
\begin{align*}
\P \left( P_0 = \begin{array}{c}\begin{tikzpicture}[scale=.3,font=\small]    
\draw (-.9,.5) node{$L$};   
\draw (2.7,-.6) node{$R$};                                
\draw (3.3,0) -- (-1.3,0);                                    
\draw (3,0) arc (0:180:1);                                     
\draw (1,0) arc (0:-180:1);    
\draw[dashed] (3,0) to [bend left=45] (-1,-2) ;
\draw[dashed] (-1,0) to [bend right=-45] (3,2) ;   
\end{tikzpicture}\end{array}\right)&
=
\P \left( P_0 = \begin{array}{c}\begin{tikzpicture}[scale=.3,font=\small]    
\draw (-.9,-.6) node{$L$};   
\draw (2.7,.4) node{$R$};                                
\draw (3.3,0) -- (-1.3,0);                                    
\draw (1,0) arc (0:180:1);                                     
\draw (3,0) arc (0:-180:1);    
\draw[dashed] (3,0) to [bend right=45] (-1,2) ;
\draw[dashed] (-1,0) to [bend left=-45] (3,-2) ;   
\end{tikzpicture}
\end{array} \right) = \frac{1}{64} \sum_{\ell_0,\ell_1 \ge 0} \Cat_{\ell_0} \Cat_{\ell_1} \binom{2\ell_0}{\ell_0} \binom{2\ell_1}{\ell_1} 2^{-4\ell_0-4\ell_1} \\
&=\frac{1}{64} \left( \sum_{\ell_0} \Cat_{\ell_0}\binom{2\ell_0}{\ell_0}  2^{-4\ell_0}\right)^2
= \frac{1}{64} \left( \frac{4}{\pi} \right)^2 =\frac{1}{4\pi^2} \approx 0.025,
\end{align*}
where the third equality was computed with Mathematica.
Hence restricting the sum in the right-hand side of \cref{prop:formulaUpper} to $k = 1$
gives a bound $1-\frac\kappa{2} \ge \frac34+ 2* 0.025 \approx 0.8$, i.e.~$\kappa \le 0.4$.
As for the lower bound, we have implemented the formula in Sage
and obtained a rigourous upper bound~$\kappa \le 0.292$.

\appendix
\section{\revision{Hasse diagram of $NC(n)$.}}
\label{sec:appendix}

\revision{We recall here the definition of the Hasse diagram of the set of non-crossing partitions $NC(n)$ of an integer $n$. We refer to \cite{GouldenMeandric} for more details and results.}

\revision{A partition of an integer $n$ is a set of blocks $\{ V_1, \ldots, V_k\}$, where all blocks are non-empty, disjoint and their union is $\{1, \ldots, n\}$. We write $V \in \pi$ to say that a block $V$ is part of a partition $\pi$. A partition is non-crossing if there is no $i \neq j \in \{1, \ldots, k\}$ and $a<b<c<d$ such that $a, c \in V_i$ and $b,d \in V_j$. We denote by $NC(n)$ the set of non-crossing partitions of $n$.}

\revision{The set $NC(n)$ is endowed with a partial order as follows: for any $\pi, \rho \in NC(n)$, we say that $\pi \leq \rho$ if, for any $V \in \pi$, there exists $W \in \rho$ such that $V \subseteq W$. In other words, $\rho$ can be obtained from $\pi$ by merging blocks. In particular, with this order, $(NC(n), \leq)$ admits a minimum element $0_n$ (the partition of $\{1, \ldots, n\}$ into $n$ blocks of size $1$), and a maximum element $1_n$ (the partition into a unique block of size $n$).}

\revision{For $\rho, \pi \in NC(n)$, we say that $\rho$ covers $\pi$ if $\pi < \rho$ and there is no $\theta$ such that $\pi <\theta <\rho$.
The Hasse diagram of $NC(n)$ is the (undirected) graph whose vertices are the elements of $NC(n)$, and where two vertices $\pi, \rho$ are connected if $\pi$ covers $\rho$ or $\rho$ covers $\pi$.}

\revision{In \cite{GouldenMeandric}, the authors construct a bijection from $NC(n)^2$ to the set of meandric systems of size $n$, as shown in Fig. \ref{fig:bijnc}. Roughly speaking,
they represent a partition $\pi$ by drawing the points $\{ 1, \ldots, n \}$ on a line, and then connecting all points in a given block of $\pi$ to the same external point above the axis, in a noncrossing way. They also do the same for a partition $\rho$ below the axis, which gives a graph $\Gamma(\pi,\rho)$. Finally, drawing one point immediately before and one immediately after each of the integers $\{1, \ldots, n \}$, they join these new $2n$ points by following the edges of the graph $\Gamma(\pi, \rho)$.
This gives a meandric system $M(\pi,\rho)$ on $2n$ points.
}

\revision{This construction is a bijection (see \cite[Section $3$]{GouldenMeandric}), with the additional property that 
\begin{align*}
d_{\mathcal{H}_n}(\pi, \rho) = n - cc(M(\pi, \rho)).
\end{align*}
}

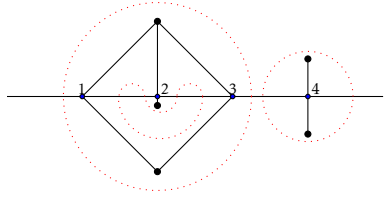
\begin{figure}
\begin{tikzpicture}
\draw (0, 0) -- (5,0) (1,0) -- (2,1) -- (3,0) -- (2,-1) -- (1,0) (2,1) -- (2,-.12) (4,.5) -- (4,-.5);
\draw[fill] (2,1) circle (.04);
\draw[fill] (2,-1) circle (.04);
\draw[fill] (2,-.12) circle (.04);
\draw[fill] (4,.5) circle (.04);
\draw[fill] (4,-.5) circle (.04);
\draw[fill=blue] (1,0) circle (.03);
\draw[fill=blue] (2,0) circle (.03);
\draw[fill=blue] (3,0) circle (.03);
\draw[fill=blue] (4,0) circle (.03);
\draw[red,dotted] (2.6,0) arc (0:180:.17);
\draw[red,dotted] (2.6,0) arc (0:-180:.56);
\draw[red,dotted] (2.26,0) arc (0:-180:.21);
\draw[red,dotted] (1.83,0) arc (0:180:.17);
\draw[red,dotted] (3.25,0) arc (0:360:1.25);
\draw[red,dotted] (4,0) circle (.6);
\node[scale=.5] at (1,.1) {$1$};
\node[scale=.5] at (2.1,.1) {$2$};
\node[scale=.5] at (3,.1) {$3$};
\node[scale=.5] at (4.1,.1) {$4$};
\end{tikzpicture}

\caption{An example of the bijection of \cite{GouldenMeandric}. Here, $n=4$, the partition $\pi$ (top) is $\{\{1,2,3\}, \{4\}\}$ and the partition $\rho$ (bottom) is $\{ \{1,3\}, \{2\}, \{4,\} \}$. The red dots represent the associated meander of size $4$.}
\label{fig:bijnc}
\end{figure}

\section{Erratum}

This is an erratum to our paper.
All statements there, and the general proof strategy, are correct,
but there is a missing argument in the proof of Proposition \ref{prop:CvToInfiniteNoodle}. Let us first recall its statement.

{\em Note:} as much as possible, we use boldface notation for random objets.

As in the previous version, we make use of Proposition \ref{prop:LocalCvDyckpath}. Applying this result to two independent pairs $(\bm w^a_n,\bm i^a_n)$
and $(\bm w^b_n,\bm i^b_n)$ give the following:
for any bounded continuous function $F: \mathcal W_* \times  \mathcal W_* \rightarrow \R_+$
\begin{equation}
\label{eq:what_we_have}
\E \left[ F ((\bm w^a_n,\bm i^a_n),(\bm w^b_n,\bm i^b_n)) \,  \middle|\,  \bm w^a_n, \bm w^b_n \right] \overset{\P}{\underset{n \rightarrow \infty}{\rightarrow}} \E \big[ F((\bm w^a_\infty, 0),(\bm w^b_\infty, 0))\big],
\end{equation}
where $\bm w^a_\infty$ and $\bm w^b_\infty$ are two independent copies
of a bi-infinite word with uniform independent letters taken in the 2-element set $\{L,R\}$.
Indeed, this follows directly from \eqref{eq:ToProve} for functions $F$
of the form $F((w_1,i_1),(w_2,i_2))=F_1(w_1,i_1) \cdot F_2(w_2,i_2)$
(called later {\em multiplicative functions})
and it is extended to all $F$ by a standard density argument.
Note that, since $F$ is bounded, moment convergence in \eqref{eq:what_we_have}
follows automatically from convergence in probability.

On the other hand, we recall
that meandric systems are in bijection with pairs of well-parenthesized words. Through this bijection, Proposition \ref{prop:CvToInfiniteNoodle} can be rewritten 
as follows: let $\bm w^a_n$ and $\bm w^b_n$ be two independent uniform well-parenthesized words
on $\{0,\dots,n-1\}$ and $\bm i_n$ be an independent random element in $\{0, \ldots, n-1\}$,
then, for any bounded continuous function $F: \mathcal W_* \times  \mathcal W_* \rightarrow \R_+$,
we have
\begin{equation}
\label{eq:to_prove}
\E \left[ F ((\bm w^a_n,\bm i_n),(\bm w^b_n,\bm i_n)) \, \middle| \, \bm w^a_n,\bm w^b_n \right] \overset{\P}{\underset{n \rightarrow \infty}{\rightarrow}} \E \big[ F((\bm w^a_\infty, 0),(\bm w^b_\infty, 0))\big]. 
\end{equation}
The difference between \cref{eq:what_we_have,eq:to_prove} is that, in the second case,
we use the same root $\bm i_n$ in both well-parenthesized words $\bm w^a_n$
and $\bm w^b_n$.
%
%

Let us introduce some notation.
For a word $w$, an index $i$ in $\{ 0, \ldots, n-1\}$ and a radius $k>0$,
we denote by $(w,i)_k$ the restriction
of $w$ to the set $[i-k,i+k]$ of positions (replacing $i-k$ by $0$ if $i-k<0$, and $i+k$ by $n-1$
if $i+k>n-1$).
Moreover, we denote by $d_{TV}$ the total variation distance. 
We will need the following \enquote{swapping} lemma.

\begin{lemma}
\label{lem:additional}
Fix $k > 0$ and $\eta>0$.
We let $(i_n^1)_{n \geq 1},(i_n^2)_{n \geq 1}$ be two integer sequences with $|i_n^1-i_n^2| \ge 2k+1$
and $\eta n \le i_n^1,i_n^2 \le (1-\eta)n$ for all $n \geq 1$.
Moreover, for all $n$, let $\bm w_n$ be a uniform well-parenthesized word on $\{ 0, \ldots, n-1 \}$.
Then the following holds, uniformly in $i_n^1,i_n^2$:
\begin{equation}\label{eq:swap_simple}
d_{TV}\big( (\bm w_n,i_n^1)_k, (\bm w_n,i_n^2)_k \big)
 \underset{n \rightarrow \infty}{\longrightarrow} 0.
\end{equation}
More generally, let $i_n^0,i_n^1, i_n^2,i_n^3$ be four integer sequences with 
$\eta n \le i_n^0,i_n^1,i_n^2,i_n^3 \le (1-\eta)n$,
and such that, for every $n$, the minimal distance between $i_n^0,i_n^1, i_n^2,i_n^3$ is at least $2k+1$. Then we have
\begin{equation}\label{eq:swap_double}
d_{TV}\Big( \big( (\bm w_n,i_n^0)_k, (\bm w_n,i_n^1)_k \big) , \big((\bm w_n,i_n^2)_k, (\bm w_n,i_n^3)_k \big)\Big)
 \underset{n \rightarrow \infty}{\longrightarrow} 0.
\end{equation}
\end{lemma}

\begin{proof}
We first prove \eqref{eq:swap_simple}.
For any word $w_n$ of size $n$, let us define $\tilde{w}_n(i_n^1, i_n^2;k)$ the word obtained from $w_n$ by swapping the neighbourhoods of radius $k$ of $i_n^1$ and $i_n^2$ in $w_n$
 (these neighbourhoods are disjoint since $|i_n^1-i_n^2| \ge 2k+1$). 
It may happen that the word that we obtain is not well-parenthesized anymore,
in which case we set $\tilde{w}_n(i_n^1, i_n^2;k) = \dagger$ instead.
A sufficient condition for $\tilde{w}_n(i_n^1, i_n^2;k)$ to be a well-parenthesized word 
is that $H(w_n,i_n^1)> 2 k$ and $H(w_n,i_n^2)> 2 k$, where $H(w,i)$ is the number of $L$'s minus the number of $R$'s among the first $i$ letters of $w$
(this is the height of the associated Dyck path after step $i$). 
It is well known that $\frac{1}{\sqrt n}H(\bm w_n,\lceil tn \rceil)$
converges to a Brownian excursion (see, e.g., \cite{aldous1993}),
which implies in particular that, for any $\eta>0$,
\begin{align*}
\P \, \big(\inf \{H(\bm w_n, i), i \in [\eta n, (1-\eta)n] \} > 2 k \big) \underset{n \rightarrow \infty}{\rightarrow} 1.
\end{align*}
Hence $\tilde{\bm w}_n(i_n^1, i_n^2;k)$ is well-parenthesized
with probability tending to 1.
Since the application $w_n \mapsto \tilde w_n(i_n^1, i_n^2;k)$ is injective on the set 
where it yields a well-parenthesized word, we have
\[ d_{TV}\big( \bm w_n, \tilde{\bm w}_n(i_n^1, i_n^2;k) \big)  \underset{n \rightarrow \infty}{\longrightarrow} 0.\]
Hence, taking the restriction to $[i_n^1-k, i_n^1+k]$, we have
\[ d_{TV}\Big( (\bm w_n,i_n^1)_k, \big( \tilde{\bm w}_n(i_n^1, i_n^2;k) ,i_n^1 \big)_k \Big)
\underset{n \rightarrow \infty}{\longrightarrow} 0.\]
But,  conditionally on $\tilde{\bm w}_n(i_n^1, i_n^2;k) \neq \dagger$ (which happens with probability tending to $1$), 
we have $\big( \tilde{\bm w}_n(i_n^1, i_n^2;k) ,i_n^1 \big)_k  =  (\bm w_n,i_n^2)_k$.
This proves \eqref{eq:swap_simple}.
\cref{eq:swap_double} is proved in a similar way, swapping simultaneously
the neighbourhoods of $i_n^0$ and $i_n^2$, and that of $i_n^1$ and $i_n^3$.
\end{proof}

We can now prove Proposition \ref{prop:CvToInfiniteNoodle}.

\begin{proof}[Proof of Proposition \ref{prop:CvToInfiniteNoodle}]
  Fix $k>0$ and $\eta>0$. Let $\bm w^a_n$ and $\bm w^b_n$ be independent uniform well-parenthesized words on $\{0,\cdots,n-1\}$
and let $i^a$ and $i^b$ be elements in $[\eta n, (1-\eta)n]$
with $|i^a - i^b| \ge 2k+1$.
From \cref{eq:swap_simple} in \cref{lem:additional},
 we have, uniformly on $i^a$ and $i^b$,
  \[d_{TV}\left(  (\bm w^b_n,i^a)_k,
(\bm w^b_n,i^b)_k  \right)
 \underset{n \rightarrow \infty}{\longrightarrow} 0.\]
 Take now $\bm i^a_n$ and $\bm i^b_n$ uniformly at random in $\{0,\cdots,n-1\}$,
 independently from $\bm w^a_n$ and $\bm w^b_n$.
 The event that both belong to $[\eta n, (1-\eta)n]$ has probability $1-4\eta$, 
 and they satisfy $|\bm i^a_{n} - \bm i^b_{n}| \ge 2k+1$ with probability tending to $1$.
We get
\[\limsup_{n \to +\infty} d_{TV}\Big( \big( \bm i_n^a, (\bm w^b_n,\bm i_n^a)_k\big) ,
\big(\bm i_n^a,(\bm w^b_n,\bm i_n^b)_k \big) \Big) \le 4 \eta.\]
Since this is valid for all $\eta >0$, the limsup in the left-hand side is an actual limit and has value $0$.

Let $F_1$ and $F_2$ be bounded 
functions $\mathcal W_* \rightarrow \R_+$
depending only on a neighbourhood of radius $k$ of the marked points.
Since $\bm w^a_n$ is independent from $(\bm w^b_n,\bm i^a_n,\bm i^b_n)$,
the total variation bound above implies that
\[
\lim_{n \to +\infty} 
\left|\E \big[ F_1 (\bm w^a_n,\bm i^a_n) \cdot F_2(\bm w^b_n,\bm i^a_n) \big]  -
\E \big[   F_1 (\bm w^a_n,\bm i^a_n) \cdot F_2(\bm w^b_n,\bm i^b_n)  \big]\right| =0.
\]

Combining this with the convergence of the expectation in \cref{eq:what_we_have}
and renaming $\bm i^a_n$ as $\bm i_n$, we get 
\begin{equation}
\label{eq:cv_exp}
\lim_{n \to +\infty}  \E \big[   F_1 (\bm w^a_n,\bm i_n) \cdot F_2(\bm w^b_n,\bm i_n) \big]
= \E \big[ F_1(\bm w^a_\infty, 0) \cdot F_2 (\bm w^b_\infty, 0) \big]. 
\end{equation}
This proves the convergence in {\em expectation} in \cref{eq:to_prove}
for multiplicative functions $F$.

To prove the convergence in probability, we need a second moment argument.
We consider 
\[\E \Big[  \E \big[ F_1 (\bm w^a_n,\bm i_n) \cdot F_2(\bm w^b_n,\bm i_n) \, \big| 
\, \bm w^a_n, \bm w^b_n \big]^2 \Big] 
=  \E \big[  F_1 (\bm w^a_n,\bm i_n) \cdot F_2(\bm w^b_n,\bm i_n) \cdot
F_1 (\bm w^a_n,\bbi_n) \cdot F_2(\bm w^b_n,\bbi_n) \big],\]
where $\bm i_n$ and $\bbi_n$ are independent uniform random variables in $\{0,\dots,n-1\}$.
We introduce two extra random variables $\bm i^b_n$ and $\bbi^b_n$, uniform in $\{0,\dots,n-1\}$,
independent of each other and of the other previously introduced random variables.
A similar argument as above, using \cref{eq:swap_double} instead of  \cref{eq:swap_simple},
shows that
\[\lim_{n \to +\infty} d_{TV}\Big( 
\big(\bm i_n, \bbi_n, (\bm w^b_n,\bm i_n)_k (\bm w^b_n,\bbi_n)_k \big), 
 \big(\bm i_n, \bbi_n, (\bm w^b_n,\bm i^b_n)_k,(\bm w^b_n,\bbi^b_n)_k \big) 
  \Big) =0. \]
Since $\bm w^a_n$ is independent from $(\bm w^b_n,\bm i_n, \bbi_n,\bm i^b_n,\bbi^b_n)$,
this implies
 \begin{multline}
 \lim_{n \to +\infty} \bigg | \E \big[  F_1 (\bm w^a_n,\bm i_n) \cdot F_2(\bm w^b_n,\bm i_n) \cdot
F_1 (\bm w^a_n,\bbi_n) \cdot F_2(\bm w^b_n,\bbi_n)\big] \\
- \E \big[  F_1 (\bm w^a_n,\bm i_n) \cdot F_2(\bm w^b_n,\bm i^b_n) \cdot
F_1 (\bm w^a_n,\bbi_n) \cdot F_2(\bm w^b_n,\bbi^b_n) \big] \bigg | =0.
\end{multline}
The latter expectation is equal to
\[\E \Big[  \E \big[ F_1 (\bm w^a_n,\bm i_n) \cdot F_2(\bm w^b_n,\bm i^b_n) \, \big| 
\, \bm w^a_n, \bm w^b_n \big]^2 \Big] \]
and converges to $\E \big[ F_1(\bm w^a_\infty, 0) \cdot F_2 (\bm w^b_\infty, 0) \big]^2$
 by \eqref{eq:what_we_have}. 
Summing up we have
\[\lim_{n \to +\infty} \E \Big[  \E \big[ F_1 (\bm w^a_n,\bm i_n) \cdot F_2(\bm w^b_n,\bm i_n) \, \big| 
\, \bm w^a_n, \bm w^b_n \big]^2 \Big]
= \E \big[ F_1(\bm w^a_\infty, 0) \cdot F_2 (\bm w^b_\infty, 0) \big]^2.\]
Combining with \eqref{eq:cv_exp}, this shows \eqref{eq:to_prove} for multiplicative functions $F$
depending on a neighbourhood of radius $k$ of the marked points.
The extension to all continuous bounded functions follow by standard arguments, see e.g. the proof of Proposition \ref{prop:LocalCvDyckpath}.
This ends the proof of \cref{prop:CvToInfiniteNoodle}.
\end{proof}

\section*{Acknowledgement}
We are grateful to Svante Janson for pointing out
 that the proof in our original paper was incomplete.


%

\bibliographystyle{bibli_perso}
\bibliography{bibli}
\end{document}